\documentclass[twoside,12pt,a4paper,leqno]{amsart}

\usepackage{amsmath,amsfonts,amsthm,amscd}

\numberwithin{equation}{section}
\newtheorem{Theorem}{Theorem}[section]
\newtheorem{Prop}[Theorem]{Proposition}
\newtheorem{Cor}[Theorem]{Corollary}
\newtheorem{Lem}[Theorem]{Lemma}
\theoremstyle{definition}
\newtheorem{defn}[Theorem]{Definition}
\newtheorem{Rmk}[Theorem]{Remark}

\newcommand{\N}{\mathbb{N}}

\newcommand{\R}{\mathbb{R}}
\newcommand{\K}{\mathbb{K}}

\newcommand{\cl}[1]{\overline{#1}}

\newcommand{\ball}[2]{\operatorname{\mathrm{B}}_{#2}\left(#1\right)}
\newcommand{\cball}[2]{\cl{\ball{#1}{#2}}}

\newcommand{\dvg}{\operatorname{\mathrm{div}}\,}

\renewcommand{\ge}{\varepsilon}
\renewcommand{\rho}{\varrho}

\newcommand{\bn}{B_{\alpha_n}}

\newcommand{\psa}{\Psi_{\e,\a}}

\renewcommand{\a }{\alpha }
\renewcommand{\b }{\beta }
\renewcommand{\d }{\delta }

\newcommand{\e }{\varepsilon}
\newcommand{\g }{\gamma}
\newcommand{\ie }{I_\e }

\newcommand{\m }{\mu }

\renewcommand{\O }{\Omega }

\newcommand{\nv }{\|_{V'}}

\newcommand{\fn }{f_{\a_n}}

\newcommand{\w}{W^{1,p}_0(\Omega)}

\newcommand{\h}{W^{1,2}_0(\Omega)}

\newcommand{\pt}{ \mathcal{P}_t}

\newcommand{\jn}{J_{\varepsilon,\alpha_n}}

\newcommand{\1}{ 2\mathcal{P}_1(\O) -1}
\newcommand{\nb }{\overline{n}}

\begin{document}

	\title[Multiple positive solutions for Benci-Cerami problem]
	{Multiple positive solutions for a $p$-Laplace Benci-Cerami type problem ($1<p<2$), via Morse theory}

	\author{Giuseppina Vannella}
	\address{\textit{Dipartimento di Meccanica, Matematica e Management\\ Politecnico di Bari\\ Via Orabona 4\\ 70125 Bari, Italy}}
	\email{giuseppina.vannella@poliba.it}



\begin{abstract}
	Let us consider the quasilinear problem
	\[ (P_\varepsilon) \ \ \left\{
	\begin{array}{ll}
	- \ge^p \Delta _{p}u  +  u^{p-1} = f(u)
	& \hbox{in} \ \Omega \\
	u>0 & \hbox{in} \ \Omega \\
	u=0 & \hbox{on} \ \partial \Omega
	\end{array}
	\right.
	\]
	where $\Omega$ is a bounded domain in $\R^N$ with smooth boundary,
	$N\geq 2$, $1< p < 2$, $\ge >0$ is a parameter and $f: \R \to \R$ is a continuous
	function with $f(0)=0$, having a subcritical growth.
	We prove that there exists $\ge^* >0$ such that, for every $\ge \in (0,
	\ge^*)$, $(P_\ge)$ has at least $2{\mathcal P}_1(\Omega)-1$
	solutions, possibly counted with their multiplicities, where
	${\mathcal P}_t(\Omega)$ is the Poincar\'e polynomial of $\Omega$.
	Using Morse techniques, we furnish an interpretation of the
	multiplicity of a solution, in terms of positive distinct solutions
	of a quasilinear equation on $\Omega$, approximating $(P_\ge)$.
\end{abstract}

\thanks
	{The author is partially supported by PRIN 2017JPCAPN Qualitative and quantitative aspects of nonlinear PDEs (MIUR), by the group GNAMPA of INdAM and by FRA2019 of Politecnico di Bari.}

\maketitle

\smallskip

\section{Introduction}

Let us consider the quasilinear elliptic problem
\[ (P_\ge) \ \ \left\{
\begin{array}{ll}
 - \ge^p \Delta _{p}u + u^{p-1} = f(u)
 & \hbox{in} \ \Omega \\
 u>0 & \hbox{in} \ \Omega \\
u=0 & \hbox{on} \ \partial \Omega
\end{array}
\right.
\]
where $\Omega$ is a bounded domain in $\R^N$ with smooth boundary,
 $N\geq 2$,  $1<p <2$,
$\ge >0$ is a parameter and $f: \R \to \R$ is a continuous
function with $f(0)=0$, having a subcritical growth.

In \cite{bencicerami}  Benci and Cerami studied $(P_\ge)$ for $p=2$, proving via Morse theory that the 
number of solutions to $(P_\ge) $ is related to the topology of $\O $.
In \cite{cvjde}    
the previous result was extended to the case $\nobreak{2\leq p<N}$.
In both cases it was proved that $(P_\ge)$ has at least $\1$ solutions, 
counted with their multiplicities (see Definition~\ref{poincare}).

\noindent Let us denote by $I_\varepsilon$ the energy functional of $(P_\ge)$.

When $p=2$, $I_\varepsilon$ is defined on the Hilbert space $\h$, so that the multiplicity of a solution $u_0$ is exactly one if $u_0$ is a nondegenerate critical point of $I_\varepsilon$, i.e. if $I_\varepsilon''(u_0)$ is an isomorphism. Moreover the 
nondegeneracy condition is generally verified, thanks to the celebrated result proved by Marino and Prodi \cite{MarPro}. 

When $p\neq 2$, as $I_\varepsilon$ is defined on $\w$ which is a Banach space,
a lot of difficulties arise in order to relate hessian notions to topological objects.
In fact, it is not clear what can be a reasonable definition of nondegenerate critical point, as it makes no sense to require that the second derivative of the energy functional in a critical
point is invertible, since a Banach space, in general (and $\w$ in particular), is not isomorphic to its dual space.
Furthermore it can be proved that 
$I_\varepsilon''(u_0)$ can not be even a Fredholm operator and Marino-Prodi perturbation type
results \cite{MarPro} do not hold (see also
\cite{Uhlen, chang1, CV} for further details). 
The multiplicity result in \cite{cvjde}, where $p\geq 2$, are proved exploiting critical groups estimates in the spirit of differential Morse relation, using a new definition in which a critical point is nondegenerate if the second derivative of the energy functional is injective (see \cite{CV}).

Moreover in \cite{cvjde}
a further perturbation result was proved, showing that $(P_\ge)$ is always close to a differential problem having at least $\1$ distinct positive solutions, which is an interpretation of the notion of the multiplicity of each solution to~$(P_\ge)$.

In this work we consider $(P_\ge)$ when  $p\in (1,2)$, which brings  additional delicate difficulties.
In fact, we can see that $u\in \w \mapsto \int_{\O}|\nabla u|^p\, dx \in \R$ is not $C^2$, thus also the energy functional $I_\varepsilon$  is not twice differentiable. 
Moreover, as $p<2$, even the nonlinearity $f$ could not be $C^1$ (see Remark~\ref{nonc1}), so that this further problem should also be managed.
Despite these difficulties, we extend the previous results when $1<p<2$, preserving the 
generality of a quite large class of nonlinearities $f$.
In order to do that, we build a convenient $C^1$ perturbation of $f$ and a class of problems approximating $(P_\ge)$, so that the corresponding energy functionals
are arbitrarly close to $I_\varepsilon$, according to a suitable norm (see Lemma~\ref{vicino}).

In this work we take advantage of recent results proved in \cite{cdv}, introducing some bilinear forms defined on a Hilbert space, which are inspired by the formal second derivatives of the approximating functionals.

The critical case of the problem, introduced by Brezis-Nirenberg \cite{brezisnirenberg} in the semilinear case  $p=2$ and  extended to the quasilinear case  $p\neq 2$ by Azorero-Peral \cite{AP1, AP2} and  Guedda-Veron \cite{gueddaveron}, was studied in \cite{cvbn} for $p\in (1,2)$,  where we proved a multiplicity result when $f$ is a homogeneous critical nonlinearity, so that the $(P.S.)$ condition at any level fails.

\bigskip

In this work, denoting by $p^*=\frac{Np}{N-p}$, we assume that 
$f\in C^0(\R)\cap C^1(\R\setminus \{0\})$ satisfies the following conditions:

\medskip

\begin{itemize}
	\item{$(f_1)$}
	there exists $q \in (p, p^*)$ such that
	\[\frac{d}{dt}\ \frac{f(t)}{t^{q-1}}<0 \quad 
	\forall \, t >0; \]
	
	\smallskip
	
	\item{$(f_2)$}
	there exists $\theta \in (0, 1/p)$ such that
	\[
	F(t) \leq \theta t f(t) \quad 
	\forall \, t \geq 0
	\]
	where $F(t)= \int_0^t f(s) \, ds$;
	
	\smallskip
	
	\item{$(f_3)$}
	${\displaystyle \frac{d}{dt}\ \frac{f(t)}{t^{p-1}}>0}$ \ \, 
	$\forall \,t>0$;
	
	\smallskip
	
	\item{$(f_4)$}
	${\displaystyle\lim_{\ t \to 0^+} t^{2-p} f'(t) =0}$;
	
	\smallskip
	
	\item{$(f_5)$}
	$f(t)=0$ 
	$\ \, \forall \, t <0$.
\end{itemize}
\bigskip

Continuity of $f$ and assumption $(f_5)$ give that $f(0)=0$.

\begin{Rmk}\label{nonc1} 
We observe that the functions satisfying the previous assumptions may not be $C^1$ in $0$.
For example, this is the case for $\ f(t)=(t^+)^{r-1}\!$, where $p<r<\min \{2,p^*\}$. In particular, note that if $N\geq 3$ and $p\in (1,\frac{2N}{N+2})$, then  $p^*<2$.
\newline Conversely, if we assume 
that $f$ is $C^1$ on $\R$, then $f'(0)=0$, which directly gives $(f_4)$.
\end{Rmk}

\begin{Rmk}\label{esempif} 
	Let us note that the assumptions $(f_1)-(f_5)$ are satisfied also by non~homogeneous functions.
	For instance, we may think of
	\[f(t) =a_1(t^+)^{r_1-1}+a_2(t^+)^{r_2-1}+\dots a_m(t^+)^{r_m-1}\]
	
	\noindent where $a_1,a_2, \dots a_m >0$ and $p<r_1<r_2<\dots r_m <p^*$.
	
	\medskip 
	
	\noindent Another example is given by
	\[f(t)=	\frac{d}{dt} \, \bigl((t^+)^r \log(a+t^+)\bigr)\]
	
	\noindent where $r\in (p, p^*)$ and $a$ is big enough. 
\end{Rmk}

\bigskip

\bigskip

In this work, inspired by the ideas in \cite{bencicerami}, we want to prove multiplicity results, related to the topology of $\O $. 
In order to estabilish the first one, we denote by $\textit{cat}_\O (\O)$ the Ljusternick--Schrnirelmann category of $\O $ in itself.

\begin{Theorem}\label{cat}
	If $\textit{cat}_\O (\O)>1$, there
	exists $\ge^*>0$ such that, for every $\ge\in (0,\ge^*)$, $(P_\ge)$
	has at least $\textit{cat}_\O (\O)+1$ distinct solutions.
\end{Theorem}

\bigskip

In order to state the following result, which will be proved exploiting Morse Theory, let us recall a classical topological definition.

\begin{defn}\label{morsepoly}
	Let $\K$ be a field. For any pair of topological spaces $(A,B)$ with $B\subset A$, we denote by ${\mathcal 
		P}_t(A,B)$ the Poincar\'e polynomial of 
	$(A,B)$, defined as 
	\[
	{\mathcal 
		P}_t(A,B)= \sum_{q=0}^{+\infty} \dim H^q (A,B)
	\, t^q
	\]
	where  $H^{q}(A,B)$ stands for the $q$-th Alexander-Spanier
	relative cohomology group of $(A,B)$, with coefficients in $\K$; we also define the Poincar\'e polynomial of $A$ 
	 as $${\mathcal P}_t(A)={\mathcal P}_t(A,\emptyset).$$
\end{defn}

\bigskip

\begin{Theorem}\label{aim}
There exists $\ge^*>0$ such that, for every $\ge\in (0,\ge^*)$, $(P_\ge)$
has at least $2\,{\mathcal P}_1(\O)-1$ solutions, possibly counted
with their multiplicities.
\end{Theorem}

\bigskip

The definition of multiplicity of a solution is given in Definition~\ref{poincare}.
 Note that, as showed in \cite{bencicerami},  $2\,{\mathcal P}_1(\O)-1$ is bigger than $\textit{cat}_\O (\O)+1$, if we assume that $\O $ is topologically rich. However the last theorem, proved applying a topological version of Morse theory, does not guarantee the existence of $2\,{\mathcal P}_1(\O)-1$ distinct solutions, so it is crucial to understand more deeply what the notion of multiplicity of a solution means.
 Indeed we prove here that there is a sequence of quasilinear problems approaching $(P_\e)$, each of them having at least $2\,{\mathcal P}_1(\O)-\!~\!1$ distinct solutions, which are close to the solutions of  $(P_\e)$. 

\noindent More precisely, we prove the following
perturbation result, in which we say that $\partial \Omega$ satisfies the
interior sphere condition if for each $x_0 \in
\partial \Omega$ there exists a ball $\ball{x_1}{R} \subset \Omega$ such
that $\cball{x_1}{R} \cap
\partial \Omega =\{x_0 \}$.

\begin{Theorem}\label{finale}
Assume that $\partial \Omega$ satisfies the interior sphere
condition.
There exists $\ge^*>0$ such that, for every $\ge\in (0,\ge^*)$, either
$(P_\ge)$ has at least $2\,{\mathcal P}_1(\O)-1$ distinct solutions
or, for every $\a_n\to 0^+$, there
exist a sequence $f_{\a_n}$ suitably approximating $f$ and a sequence $h_n \subset C^1(\overline \O)$ with
$\|h_n\|_{C^1\!(\overline{\O})} \to 0$ such  that problem
\[ \ (P_n)\ \left\{
\begin{array}{ll}
 - \ge^p div \bigl( (|\nabla u|^2 + \alpha_n)^{(p-2)/2} \nabla u \bigr)
 +
 u\,(\a _n+u^{2})^{(p-2)/2
 }&  \\
  
   =\fn (u) +h_n
 & \hbox{in} \ \Omega \\

 u>0 & \hbox{in} \ \Omega \\
u=0 & \hbox{on} \ \partial \Omega
\end{array}
\right. 
\]
has at least $2\,{\mathcal P}_1(\O)-1$ distinct solutions,   
for $n$ large enough.
\end{Theorem}

\begin{Rmk}
Considering the  case in which $(P_\e)$ has fewer then $2\,{\mathcal P}_1(\O)-1$ distinct solutions, we will see that all the solutions of $(P_n)$ are arbitrarly close to solutions of~$(P_\e)$.
More precisely, if $\bar{u}$ is a solution of $(P_\e)$ and the multiplicity of $\bar{u}$ is $\bar{k}$, then, for any fixed $R>0$, each problem $(P_n)$ has at least $\bar{k}$ distinct solutions in $B_R(\bar{u})$ and besides these solutions converge to $\bar{u}$ in $C^1(\overline{\O})$-norm, for $n\to \infty$.
\end{Rmk}

\smallskip
We mention that in \cite{alves}, using
Ljusternik--Schnirelman category, Alves  has proved the existence of
${ cat}(\Omega)$ solutions to $(P_\varepsilon)$, when $p\geq 2$.

Perturbation results in Morse theory for quasilinear problem having
a right-hand side subcritically at infinity have been obtained in
\cite{cvmp,CLV} (see also \cite{CD,DL}).

\section{Proofs of Theorems \ref{cat} and \ref{aim}}

Combinining $(f_1)$ and $(f_4)$, we see that
there are $q\in(p,p^*)$ and $c>0$ such that, for every $t>0$,
\begin{equation}\label{somf1}
tf'(t)\leq \frac{p-1}{2}t^{p-1}+ct^{q-1}
\end{equation}
\begin{equation}\label{somf}
f(t)\leq \frac{1}{2}t^{p-1}+ct^{q-1}
\end{equation}
\begin{equation}\label{somF}
F(t)\leq \frac{1}{2p}t^{p}+ct^{q}.
\end{equation}

Standard arguments prove that the solutions to $(P_\ge)$ correspond to 
critical points of the $C^1$ functional $I_\ge: W^{1,p}_0(\Omega)
\to \R$ defined by setting
\[I_\ge(u) = \frac{\ge^p}{p} \int_{\Omega} |\nabla u|^p \, dx +
\frac{1}{p} \int_{\Omega} |u|^p \, dx -  \int_{\Omega} F(u) \, dx.\]

\noindent We define an equivalent norm on $\w$ as

\[\|u \|_\varepsilon= \left(\ge^p\int_{\Omega} |\nabla u|^p \, dx + \int_{\Omega} |u|^p \, dx\right)^{\frac{1}{p}}\]

while $\langle \cdot , \cdot  \rangle :
W^{\!-1,p'}(\Omega) \times W_0^{1,p}(\Omega) \to \R$ denotes the
duality pairing.

\bigskip

Denoting by $A_\varepsilon:W^{1,p}_0(\Omega)\to \R$
\[A_\varepsilon(u)=\langle
I'_\ge(u), u \rangle,\]
we introduce the  Nehari manifold
\[ \Sigma_\ge(\Omega) =\{u \in \w  \ : \ u \neq 0, \ A_\varepsilon(u)=0 \}. \]

Naturally each nontrivial critical point of $I_\ge$ is a nonnegative
function which belongs to $\Sigma_\ge(\Omega)$.

Although $f$ may be not $C^1$ in $0$, the
assumptions  
on $f$, through (\ref{somf1}) and (\ref{somf}),
assure that
$A_\ge$ is still a $C^1$ functional   
and 
\[\langle
A'_\ge(u), v \rangle=p\,\ge^p \int_{\Omega} \frac{(\nabla u/\nabla v)}{|\nabla u|^{2-p}} \, dx +
 p\int_{\Omega} \frac{ u v}{\ |u|^{2-p}} \, dx -  \int_{\Omega} f'(u)uv+f(u)v \ dx.\]

\medskip

The following Lemma lists some useful properties about $\Sigma_\ge(\Omega)$, which hold also when $\O$ is replaced by another bounded set or by $\R^N$. The proof is strongly inspired by \cite{bencicerami}, even if with some slightly new arguments.

\begin{Lem}\label{diff}
For every $\ge>0$, $\Sigma_\ge(\Omega)$ is a $1$-codimensional
submanifold of $\w$, which is $C^{1}$-diffeomorphic to
\[{\mathcal S}_\varepsilon =\{u\in \w : \
\|u\|_\varepsilon=1\} \setminus \{u\in \w: \ u\leq 0 \, \hbox{ a.e. in } \O \}. \]

Furthermore
there exist $\sigma_\ge>0$ and $K_\ge>0$ such that 
\begin{equation}\label{ccc}
 \|u\| \geq \sigma_\ge, \quad I_\ge(u) \geq K_\ge \qquad \forall \, u \in \Sigma_\varepsilon(\Omega).
\end{equation}
\end{Lem}
\begin{proof}
Taking account of  $(\ref{somF})$, there is $c_{\varepsilon}>0$ such that
	\begin{equation}\label{daf4}
		I_\ge(u)  \geq \frac{1}{2p}\|u\|^p_{\varepsilon}-c_{\varepsilon}\|u\|^q_{\varepsilon} \qquad \forall \, u\in \w,
	\end{equation}
	therefore $0$ is a local minimum for $I_\ge$.
	
	\medskip
	
Let us denote by
\[\O_{u,\delta}=\left\lbrace x\in \O \ :\ u(x)>\delta\right\rbrace 
\]
for any $u\in \w$ and $\delta \in \R$.
From $(f_5)$ we infer that
\[ u \in \Sigma_\ge(\Omega) \ \Longrightarrow \
|\O_{u,0}|>0,
\]

hence, by $(f_3)$,
\begin{equation}\label{dini}
\langle A'_\ge(u), u \rangle < 0 \qquad \forall u \in
\Sigma_\ge(\Omega).
\end{equation}

\smallskip

For every fixed $v \in {\mathcal S}_\varepsilon$, let us consider the map
 $t \in [0, +\infty) \mapsto I_\ge(t v)$.
 We start proving that
 \begin{equation}\label{-in}
 \lim_{t \to +
 	\infty} I_\ge(tv) = - \infty
 \end{equation}
 Actually, there is $\d _0>0$ such that $|\O_{v\!,\d_0}|>0$. Indeed, if not, for every $n\in \N$, it should be
 $|\O_{v,\frac 1 n}|=0$, so that
 \[ |\O_{v,0}|=\lim\limits_{n \to \infty}|\O_{v,\frac 1 n}|=0  \]
 which gives a contradiction, as $v\in {\mathcal S}_\varepsilon$.
 \newline By $(f_2)$ we infer that
 $F(s)\geq s^{\frac 1 \theta}\,F(\d_0)/\d_0^{\frac 1 \theta}$, for any $s\geq \d_0$.
 Hence, for any $t\geq 1$
 \[ I_\ge(tv)
 \leq \frac{t^p}{p}- \int\limits_{\O_{v\!,\d_0}}\hspace{-2mm}F(tv(x))\, dx \ \leq  \frac{t^p}{p}-t^{\frac 1 \theta}F(\d_0) |\O_{v,\d_0}|\,\]

 which proves (\ref{-in}).

 As a consequence, there is  $\xi>0$
 such that \[ I_\ge(\xi v)=\max_{t > 0}I_\ge(t v).\]
  Clearly $\xi v$ belongs to $\Sigma_\ge(\Omega)$ and $\int_{\O}
  f(\xi v)v/{\xi}^{p-1} \ dx=1,$ hence we
deduce  by $(f_3)$ that $\xi=\xi_\ge(v)$ is unique
 and $\Sigma_\ge(\Omega)$  is the image of the function $\psi_\ge :
 {\mathcal S}_\varepsilon \to
 \w$ defined by  $\psi_\ge(v) =
 \xi_\ge(v) v$.

Taking account of (\ref{dini}), the implicit function theorem assures that $\xi_\ge$ and $\psi_\ge$ are $C^{1}$ 
functions.

\smallskip

Finally, for each $u\in \Sigma_\ge(\Omega)$,  $v=u/\|u\|_\varepsilon$ belongs to $ {\mathcal S}_\varepsilon$ and, using (\ref{daf4}),
\[I_\ge(u)= \max_{t \geq 0}I_\ge(tv)\geq \max_{t \geq 0} \left(\frac{1}{2p}t^p -c_\varepsilon t^q\right)=K_\varepsilon>0.\]

As $I_\ge(0)=0$, by continuity we complete the proof.
\end{proof}

\bigskip

The following Lemma shows how $\Sigma_{\ge}(\Omega)$ is a natural constraint for problem $(P_\ge)$.

\begin{Lem}
	$u$  is a nontrivial
	critical point of $I_\ge$ if and only if it is a
	critical point of $I_\ge$ on $\Sigma_{\ge}(\Omega)$, moreover $(I_\ge)$ and $(I_\ge)_{|\Sigma_\ge(\Omega)}$ satisfy $(P.S.)_c$ for
	all $c \in \R$.
\end{Lem}

\begin{proof}
	The first statement comes directly from (\ref{dini}).
	\newline In Corollary \ref{lps} it will be proved that $I_\ge$ satisfies $(P.S.)_c$ for any $c\in \R$.
	\newline Let $c\in \R$ and $u_k\subset \Sigma_{\ge}(\Omega)$, $\lambda_k \subset \R$ be sequences such that $\ie(u_k)\to c$ and 
	\begin{equation}\label{vinc1}
	\ie'(u_k)-\lambda_k A_\ge'(u_k)\to 0. 
	\end{equation}
	Since
	$\left(\frac 1 p-\theta\right)\|u_k\|_\e^{p}\leq \ie(u_k)
	$, the sequence
	$u_k$ is bounded, so that
	\begin{equation}\label{vinc0}
		-\lambda_k\langle A_\ge'(u_k),u_k\rangle \to 0.
		\end{equation}

	Moreover there is $\bar u\in\w$ such that
	$u_k$ converges to $\bar u$, weakly in $\w$ and strongly in $L^r(\O)$, if $r\in (p,p^*)$.
		Therefore, through (\ref{somf1}) and (\ref{somf}),
	\medskip
	
	\[
	\int_\O f'(u_k)u_k^2-(p-1)f(u_k)u_k\ \to \int_\O f'(\bar u)\bar u^2-(p-1)f(\bar u)\bar u\]
	
	where, by $(f_3)$, $a_0=\int_\O f'(\bar u)\bar u^2-(p-1)f(\bar u)\bar u \geq 0$. If $a_0=0$,
	then $\bar u(x)\leq 0$ almost everywhere in $\O$, and in particular
	
	\[\|u_k\|_\varepsilon^p=\int_\O f(u_k)u_k\, dx \ \to \ \int_\O f(\bar u)\bar u\, dx=0\]
	which contradicts (\ref{ccc}).

	So, taking account of (\ref{vinc0}) and (\ref{vinc1}),
	\[
	-\langle A_\ge'(u_k),u_k\rangle\ \to \ a_0>0 \ \Rightarrow \ \lambda_k\to 0  \ \Rightarrow \ \ie'(u_k)\to 0
	\]
	
	which, as  $I_\ge$ satisfies $(P.S.)_c$, concludes the proof.
\end{proof}

\bigskip

Since $I_\ge$ satisfies (P.S.) on $\Sigma_\ge(\Omega)$, the infimum is
achieved. Let us denote 
\[m(\e, \Omega)= \inf\{ I_\varepsilon (u) \ : \ u \in
\Sigma_\ge(\Omega) \}.
\]

\bigskip

Without any loss of generality, we shall assume that $0 \in \Omega$.
Moreover we denote by $r>0$ a number such that $\Omega_r^+ = \{ x
\in \R^N \ | \ d(x, \Omega) < r \}$ and $\Omega_r^- = \{ x \in
\Omega \ | \ d(x, \partial\Omega) > r \}$ are homotopically
equivalent to $\Omega$ and $\ball{0}{r} \subset \Omega$.

\noindent We notice that if $\Omega = \ball{y}{r}$, the number $m(\varepsilon,
\ball{y}{r})$ does not depend on $y$, so we set
\[m(\varepsilon, r) = m(\varepsilon, \ball{y}{r}).\]

We also set  $\Sigma_\varepsilon^{m(\varepsilon,r)} =\{ u \in
\Sigma_\varepsilon (\Omega)  \ : \ I_\varepsilon(u) \leq
m(\varepsilon, r) \}$.

\bigskip

Now we can reason as in \cite{cvjde} so that, relying also on \cite{doomed,franchilanconelliserrin} which still hold when $p\in (1,2)$, we infer the following two results (cf. Proposition 4.4 and 4.6 in \cite{cvjde}).

\bigskip

\begin{Prop}\label{map}
	There exists $\varepsilon^* >0$ such that for any $\varepsilon \in
	(0, \varepsilon^*)$
	\[
	\dim H^k(\Sigma_{\varepsilon}^{m(\varepsilon, r)} )\geq \dim
	H^k(\O).
	\]
\end{Prop}

\bigskip

\begin{Prop}\label{ab}
	There exists $\e^*\!>0$ such that for every $\e \in (0,\e^*)$ there are
	$\nobreak{\a>m(\e, \Omega)}$ and $c\in \bigl(0,m(\e,\Omega)\bigr)$ such that
	
	\begin{equation}\label{pt1}
	\pt(I_\e^{\a},I_\e^c)=t\pt(\O)+t\mathcal{Z}(t)
	\end{equation}
	\begin{equation}\label{pt3}
	\pt(\w,I_\e^{\a})=t^2\bigl(\pt(\O)-1\bigr)+t^2\mathcal{Z}(t)
	\end{equation}
	where $\mathcal{Z}(t)$ is a polynomial with nonnegative integer
	coefficients.
\end{Prop}

\bigskip 

\noindent {\bf Proof of Theorem~\ref{cat}.\enspace} By Proposition \ref{map}, we infer that \[\textit{cat}_{\Sigma_\ge^{m(\varepsilon, r)}} (\Sigma_\ge^{m(\varepsilon, r)})\geq \textit{cat}_\O (\O),\] so that, applying classical results of Ljusternick-Schnirelmann theory,  $I_\e:\Sigma_\ge^{m(\varepsilon, r)}\to\R$ has at least $\textit{cat}_\O (\O)$ critical points. 
Moreover, having assumed $\textit{cat}_\O (\O)>1$, we have that $\Sigma_\ge^{m(\varepsilon, r)}$ is not contractible, while $\Sigma_\ge(\O)$ is, hence there is a further critical point $u$ with $I_\e(u)>m(\varepsilon, r)$. $\hfill 
\Box$

\bigskip

In order to prove Theorem~\ref{aim}, which involves Morse theory, 
we recall some notions (see \cite{chang,chang3}).

\begin{defn}\label{gruppi}
	Let $\K$ be a field, $X$ a Banach space and $f$ a $C^1$ functional on $X$. Let $u$ be a critical point of $f$, $c=f(u)$ and
	$U$ be a neighborhood of $u$. We call
	\[
	C_{q}( f,u) = H^{q}( f^{c} \cap U,(f^{c}\setminus\{ u\}) \cap U)
	\]
	the q-th critical group of $f$ at $u$, $q=0,1,2,\dots$, where $f^c=
	\{v \in X \, : \, f(v) \leq~c \}$,  $H^q(A,B)$ stands for the $q$-th
	Alexander-Spanier cohomology group of the pair $(A,B)$ with
	coefficients in  $\K$. By the excision property of the singular
	cohomology theory, the critical groups do not depend on a special
	choice of the neighborhood~$U$.
\end{defn}

\begin{defn}\label{poincare}
	We introduce the Morse polynomial of $f$ in $u$,
	defined as
	\[
	i(f,u)(t) = \sum_{q=0}^{+\infty} \dim C_q (f,u) \, t^q.
	\]
	We call {\sl multiplicity} of $u$ the number $i(f,u)(1) \in \N \cup \{+ \infty\}$.
\end{defn}

The following theorem is a topological version of the classical Morse
relation (cf. Theorem 4.3 in \cite{chang}).

\begin{Theorem}\label{morse-topl}
	Let $X$ be a Banach space and $f$ be a $C^1$ functional on $X$. Let
	$a,b \in \R$ be two regular values for $f$, with $a<b$. If $f$
	satisfies the $(P.S.)_c$ condition for all $c\in (a,b)$ and $u_1 , \dots , u_l$ are the critical points of $f$ in $f^{-1}(a,b)$, then
	
	\begin{equation}\label{morse-rel}
	\sum_{j=1}^{l}i(f,u_j)(t)
	={\mathcal P}_t(f^b,f^a)+(1+t)Q(t)
	\end{equation}
	where $Q(t)$ is a formal series with coefficients in $\N \cup \{+
	\infty\}$.
\end{Theorem}

\bigskip

\bigskip

\noindent {\bf Proof of Theorem~\ref{aim}.\enspace} Choosing $\e^*$
as required by Proposition \ref{ab}, the proof comes from
(\ref{morse-rel}), (\ref{pt1}) and (\ref{pt3}). In particular,
denoting by $m(u)$ the multiplicity of any critical point $u$ of
$I_\e$ (see Definition~\ref{poincare}), we get

\[\sum _{I_\e(u)<\a}m(u)= {\mathcal{P}_1}(\O)+\mathcal{Z}(1)+2Q_-(1)\geq {\mathcal{P}_1}(\O)\]

\[\sum _{I_\e(u)>\a}m(u)= {\mathcal{P}_1}(\O)-1+\mathcal{Z}(1)+2Q_+(1)\geq {\mathcal{P}_1}(\O) -1\]
where, by Theorem~\ref{morse-topl}, $Q_-(t)$ and $Q_+(t)$ are suitable formal series with
coefficients in $\N \cup \{+ \infty\}$. $\hfill \Box$

\bigskip

\section{Approximating functionals}

In order to obtain a further result which provides at least $2{\mathcal{P}_1}(\O)-1$ distinct solutions for a sequence of problems approaching $(P_{\e})$, we build here some functionals which approximate $I_{\e}$.

\bigskip

Let us fix 
\begin{equation}\label{defs}
s>\max \left\lbrace 2,2q-1\right\rbrace
\end{equation}
\indent where $q$ is introduced by $(f_1)$. For  every $\a \geq 0$ we set
\[	F_\a(t)=F\left(\left(\a+(t^+)^s\right)^{1/s}\right),\]
\[f_\a(t)=F'_\a(t),\]
\[G_\a(t)=\frac{1}{p} \bigl(\alpha+t^{2}\bigr)^{\frac p {2}}.\]

For every $\e,\, \a > 0$ and $h\in C^1(\overline{\,\Omega})\ $
we define 

\begin{equation} \label{ja}
J_{\varepsilon,\alpha}(u)= \displaystyle{\frac{\e^p}{p} \int_{\Omega}} \bigl(\a +|\nabla u|^2\bigr)^{\frac p 2} dx
+\displaystyle{\frac{1}{p} \int_{\Omega} \bigl(\alpha+u^{2}\bigr)^{\frac p {2}} dx
	-\, \int_{\Omega}} F_{\a} (u)dx
\end{equation}

\[
J_{\e ,\alpha,h}(u)=J_{\e ,\alpha}(u)-\int_{\Omega} h(x) u(x)\,dx.
\]

\bigskip

It is immediate that
\[
	I_{\varepsilon}(u)\!= \!\!\int_{\Omega}\! \!\e^p G_0(|\nabla u|) \!+\! G_0(u)\!-  \!F_{0} (u)\quad 
	J_{\varepsilon,\alpha}(u)\!=\!\! \int_{\Omega}\!\!\e ^pG_\a(|\nabla u|)\!+\! G_\a( u)\!	- \! F_{\a} (u).
\]

Note that $\, G_{\a}, F_{\a} \in C^2(\R, \R)$
when $\a>0$, while $G_0$ is just $C^1$ and so it could be about $F_0$ (see Remark \ref{nonc1}).

Nevertheless,  the functional $\ u\mapsto \int_{\Omega}  G_\a(|\nabla u|)$ and, consequently, $J_{\varepsilon,\alpha}$ are still just $C^1$ in  $\w$.

\bigskip 

\begin{Lem}\label{vicino}
	For any bounded $B\subset \w$
	\begin{equation}\label{vicinalfa}
		\lim\limits_{\a \to 0}\|J_{\varepsilon,\alpha}-I_{\varepsilon}\|_{C^1(B)}=0
	\end{equation}
	\begin{equation}\label{vicinh}
	\lim\limits_{\|h\|_{C^1\!(\bar{\O})} \to 0}\|J_{\varepsilon,\alpha,h}-J_{\varepsilon,\alpha}\|_{C^1(B)}
	=0.
	\end{equation}

\end{Lem}
\begin{proof}
We observe that, for any $t\in \R$
\begin{equation}\label{ga}
	|G_\a (t)-G_0(t)|=\left|\frac 1 p (\a +t^2)^{p/2}-\frac 1 p|t|^p\right|\leq \frac{\ \a^{p/2}}{p}\, ,
\end{equation}

\smallskip

\begin{equation}\label{g1a}
\begin{split}
|G'_0 &(t)-G'_\a(t)| 
\,=\,|t|^{p-1}\frac{(\a +t^2)^{\frac{2-p}{2}}-|t|^{2-p}}{(\a +t^2)^{\frac{2-p}{2}}}\\
\smallskip
&\leq\frac{|t|^{p-1}\a^{\frac{2-p}{2}}}{(\a +t^2)^{\frac{2-p}{2}}}\ \leq  \begin{array}{ll}
\medskip
\a^{\frac{p-1}{2}}
& \hbox{if} \ p\in (1,\frac 3 2] \\
\smallskip
\a^{\frac{2-p}{2}}|t|^{2p-3} & \hbox{if} \ p\in (\frac 3 2, 2).
\end{array}
\end{split}
\end{equation}

\bigskip

Moreover, by $(f_2)$
$\ F(\a^{1/s})\leq \theta f(\a^{1/s})\a^{1/s},
$ hence taking account of (\ref{somf}) we get
\begin{equation}\label{fa}
	|F_\a (t)-F(t)|\leq \a^{1/s}f(|t|+\a^{1/s})
	\leq O(\a^{1/s})(1+|t|^{q-1})
	\qquad \forall \, t\in \R.
\end{equation}  

\bigskip

Having chosen $s>2q-1$, we have $\frac{s-1}{2}>q-1$, so by $(f_1)$ 

\begin{equation}\label{hdec}
	t\in (0,+\infty)\mapsto \frac{f^2(t)}{t^{s-1}} \qquad \hbox{is decreasing}.
\end{equation}

Moreover, introducing $k:(0,+\infty) \to \R\ $ defined as  $k(t)=\displaystyle{\frac{t^{\frac{s-1}{s}}}{f(t^{1/s})}}$, it is immediate that 
\[k(a)<k(a+b)<k(a)+k(b)\]
 for any $a,b>0$, hence by (\ref{hdec}) we get

\begin{equation}\label{f1a}
\begin{split}
		|F'& (t)-F_\a'(t)|=
	\Bigl| f(t)-f\left((\a +t^s)^{1/s}\right)\frac{t^{s-1}}
	{\ \left(\a +t^s\right)^{\frac{s-1}{s}}}\Bigr| \\
	&=
	\frac{\ f\left((\a +t^s)^{1/s}\right)f(t)}{\left(\a +t^s\right)^{\frac{s-1}{s}}}
	\left|\frac{\left(\a +t^s\right)^{\frac{s-1}{s}}}{f\left((\a +t^s)^{1/s}\right)}
	-\frac{\ t^{s-1}}{f(t)}\right|\\
	&=
	\frac{\ f\left((\a +t^s)^{1/s}\right)f(t)}{\left(\a +t^s\right)^{\frac{s-1}{s}}}\ 
	\Bigl(k(\a +t^s)-k(t^s)\Bigr)
	 \ \leq \ \frac{\ f\left((\a +t^s)^{1/s}\right)f(t)}{\left(\a +t^s\right)^{\frac{s-1}{s}}}\ k(\a)\\
	&\leq \ \frac{f^2\left((\a +t^s)^{1/s}\right)}{\left(\a +t^s\right)^{\frac{s-1}{s}}}
	\  k(\a)
	\,\leq \ \frac{f^2\left(\a ^{1/s}\right)}{\a^{\frac{s-1}{s}}}\ \frac{\a^{\frac{s-1}{s}}}{f(\a^{1/s})}\,=\ f\left(\a ^{1/s}\right).
	\end{split}
\end{equation}

\bigskip

So, from (\ref{ga}), (\ref{g1a}), (\ref{fa}) and (\ref{f1a}), we infer (\ref{vicinalfa}), while (\ref{vicinh}) is trivial.

\end{proof} 

\bigskip

We now aim to prove that, for every $\e>0$, $\a\in [0,1]$ and $h\in C^1(\bar{\Omega})$,  $J_{\e,\alpha,h}$ satisfies a
compactness condition. We begin to recall a classical definition in a reflexive Banach space, taken from~\cite{browder1983, skrypnik1994}.

\begin{defn}
	Let $X$ be a reflexive Banach space and $D\subset X$. A map $H:D\to X'$ is said to be of class $(S)_+$, if, for every sequence $u_k$ in $D$ weakly convergent to $u$ in $X$ with
	\[ \limsup_{k\to \infty}\langle H(u_k), u_k-u\rangle\leq 0,\]
	we have $\|u_k-u\|\to 0$.
\end{defn}

\bigskip
The following result provides a compactness property about the approximating functionals $J_{\e,\alpha,h}$. It is based on \cite[Theorem 3.5]{ad} (see also \cite[Theorem~2.1]{CD}).
For reader's convenience, we sketch the proof.

\begin{Theorem}\label{s+}
	
	For every $\varepsilon>0,\ p\in (1,2), \ \a \in [0,1],\ h\in  C^1(\bar{\O})$, the functional
	$J'_{\e,\alpha,h}$ is of class $(S)_+$.
\end{Theorem}

\medskip

\begin{proof}
	Let us fix $\e >0$.
	For every $\a \in [0,1]$, let  $\Psi_{\e,\a}:\R^N\to \R$, $b_\a:\R\to \R$ and $H_{\a}:\w\to W^{-1,p'}$ be the maps
	\[\Psi_{\e,\a}(\xi)=\e ^pG_\a(|\xi|)\] 
	\begin{equation}\label{psi+l}
	 b_\a(t)=G'_\a(t)-F'_\a(t)	
	\end{equation}
	\[H_{\a}(u)=\langle J'_{\e,\alpha}(u), \ \cdot \ \rangle\]
	
	so that $H_{\a}(u)=-\dvg\left(\nabla \psa(\nabla u)\right)+b_\a(u)$.  
	\newline We start by showing that there is $C>0$ and, for every $\delta>0$, there is a suitable $c(\delta)\in \R$ such that
	
	\smallskip
	
	\begin{equation}\label{u1}
	|\nabla \psa(\xi)|\leq \e^p|\xi|^{p-1}
	\end{equation}
	
		\begin{equation}\label{u2}
	|b_\a(s)|\leq C+C|s|^{p^*-1}
	\end{equation}
	
		\begin{equation}\label{u3}
	\nabla \psa(\xi)\cdot\xi\geq \frac{\e^p}{2}|\xi|^{p}-C
	\end{equation}
	
		\begin{equation}\label{u4}
	b_\a(s)s\geq -\delta|s|^{p^*}+c(\delta)
	\end{equation}
	
		\begin{equation}\label{u5}
	\bigl(\nabla \psa(\xi)-\nabla \psa(\eta)\bigr)\cdot \bigl(\xi-\eta\bigr)>0
	\end{equation}

	for every $\a \in [0,1],\ \xi \in \R^N, \ \eta\neq\xi \in \R^N,\ s \in \R.$
	
	As $(\ref{u1})$ and $(\ref{u2})$ are trivial, let us prove $(\ref{u3})$.
	Denoting by $\g_\a=(\frac{\a}{2^{2/(2-p)}-1})^{1/2}$,
	\[|\xi|\geq \g_\a \ \Rightarrow \
	\left(\frac{|\xi|^2}{\a + |\xi|^2}\right)^{\frac{2-p}{2}}\geq \frac 1 2
	 \quad \Rightarrow \quad \frac{\e^p|\xi|^2}{(\a+|\xi|^2)^{\frac{2-p}{2}}}\geq \frac{\e^p}{2}|\xi|^{p}\]
	
	 \[|\xi|< \g_\a \ \Rightarrow \quad \frac{\e^p}{2}|\xi|^p\leq \frac{\e^p}{2}\g_\a^p \leq \frac{\e^p}{2}\g_1^p, \]
	
	 so we infer $(\ref{u3})$, choosing $C\geq \frac{\e^p}{2}\g_1^p$.
	
	 \medskip
	
	 Now we immediately see that there is $c_1>0$ such that
	 \[|t b_\a(t)|\leq c_1(1+|t|^p+|t|^q) \quad \forall \a\in [0,1],\ t\in \R.\]
	
	So we get $(\ref{u4})$ putting \[c(\delta)=\min_{t\geq 0}\left(\,\delta|t|^{p^*}-c_1(1+|t|^p+|t|^q)\ \right).\]
	
	\medskip
	
	Let us consider $\eta\neq\xi \in \R^N$. If $|\xi|=|\eta|$, then
	\[\bigl(\nabla \psa(\xi)-\nabla \psa(\eta)\bigr)\cdot \bigl(\xi-\eta\bigr)=\e^p\,\frac{|\xi-\eta|^2}{(\a+|\eta|^2)^{\frac{2-p}{2}}}>0.\]
	
	Otherwise, if $|\xi|\neq|\eta|$, by the monotonicity of the real function $ t\in \R \mapsto t\left(\a+t^2\right)^{\frac{p-2}{2}}\!\!$, we get
	
	\[ |\xi||\eta|\left(\frac{1}{(\a+|\xi|^2)^{\frac{2-p}{2}}}+ \frac{1}{(\a+|\eta|^2)^{\frac{2-p}{2}}}\right)<\frac{|\xi|^2}{(\a+|\xi|^2)^{\frac{2-p}{2}}}+\frac{|\eta|^2}{(\a+|\eta|^2)^{\frac{2-p}{2}}}\]
	
	\smallskip
	
	which gives $(\ref{u5})$.
	
	\medskip

	As $(\ref{u1})-(\ref{u5})$ hold, $\nabla \psa$ and $b_\a$ satisfy the assumptions required by Theorem 3.5 in \cite{ad}, so that $J'_{\e,\alpha}$ is
	of class $(S)_+$. Moreover, for every $h\in C^1(\overline{\O})$, it is immediate that
	$J'_{\e,\alpha,h}$ is
	of class $(S)_+$ too.
\end{proof}

\smallskip

\begin{Cor}\label{lps}
	For every $\varepsilon>0,\ p\in (1,2), \ \a \in [0,1],\ h\in  C^1(\bar{\O})$, the functional
	$J_{\e,\alpha,h}$  satisfies $(P.S.)_c$ for all $c\in \R$.
\end{Cor}

\begin{proof}
 By (\ref{defs}) and ($f_1$), $\ t\in (0,+\infty)\mapsto \frac{f(t)}{t^{s-1}}$ is a decreasing function, hence  $(f_2)$ implies that
\[F_\alpha(t)-\theta F_\alpha'(t)t\leq \theta\alpha^{\frac 1 s}f(\alpha^{\frac 1 s}) \qquad \hbox{for any } t\in \R.\]
Moreover, there is $c_1>0$ such that $\int\limits_{\O}h u\, dx\leq c_1\|u\|_\e$ and
\begin{equation}\label{disps}
\bigl(\frac 1 p -\theta\bigr)\|u\|_\e^{p}
\leq
J_{\e,\alpha,h}(u)-\theta \langle J'_{\e,\alpha,h}(u),u\rangle
+\theta\alpha^{\frac 1 s}f(\alpha^{\frac 1 s})|\O |+(1-\theta)c_1\|u\|_\e
\, .
\end{equation}
Let $c\in \R$ and $\{u_k\}$ be a sequence such that $J_{\e,\alpha,h}(u_k)\to c$ and $J'_{\e,\alpha,h}(u_k)\to 0$.
If there is $\beta >0$ such that, up to subsequences, $\|u_k\|_\e\geq \beta$, then by $(\ref{disps})$
\[\bigl(\frac 1 p -\theta\bigr)\|u_k\|_\e^{p-1}\leq\frac{\ c+\theta\alpha^{\frac 1 s}f(\alpha^{\frac 1 s})|\O |}{\beta}+(1-\theta)c_1+o(1)\]
so $\{u_k\}$ is bounded and the previous Theorem completes the proof.
\end{proof}

\bigskip

Let us state a regularity result (see \cite{gueddaveron, lieberman, cdv} and references therein).

\begin{Theorem}\label{c1reg}
	Let $B$ be a bounded subset of $\w$ and $\e >0$.
	There exist $\eta \in (0,1)$ and $K>0$ such that, for any$ \ \a \in [0,1]$ and $h\in C^1_0(\O)$ with $\|h\|_{C^1(\bar{\O})}\leq 1$,  if $u\in B$ solves
	\[- \ge^p div \bigl( (|\nabla u|^2 + \alpha)^{(p-2)/2} \nabla u \bigr)
	+
	u\,(\a+u^{2})^{(p-2)/2
	}
	=f_\a (u) +h(x)
	\]
	 then $u \in C^{1,\eta}(\bar{\O})$ and $\|u\|_{C^{1,\eta}(\bar{\O})}\leq K$.
\end{Theorem}

\bigskip

Now let us consider a critical point $u_0$ of $J_{\varepsilon,\a ,h}$. Assume that $\a$ and $h$ satisfy the assumptions of the previous theorem, so that $u_0 \in C^{1,\eta}(\bar{\O})$, for some $\eta \in (0, 1)$. 

\noindent It is crucial to give a notion of Morse index, which is not standard, as $J_{\varepsilon,\a ,h}$ is not  $C^{2}\!$.
If $\a>0$ and $u \in W^{1,\infty}(\O)$, let us define on $W^{1,2}_0(\O)$  the following bilinear form

\[B_\a(u)(z_1,z_2)=\int\limits_{\O}\Psi_{\e,\a}''(\nabla u)[\nabla z_1,\nabla z_2] +\int\limits_{\O}b_\a'(u) z_1\,z_2\]

where $\Psi_{\e,\a}$ and $b_\a$ are defined by $(\ref{psi+l})$,
hence 

\[
\begin{split}
&\hspace{2cm}B_\a(u)(z_1,z_2)=\e ^p\int\limits_{\O}\frac{(\nabla z_1/\nabla z_2)}{(\a +|\nabla u|^2)^{\frac{2-p}{2}}}\\
&\hspace{1.5cm}-\e ^p(2-p)\int\limits_{\O}\frac{(\nabla u/\nabla z_1)(\nabla u/\nabla z_2)}{(\a +|\nabla u|^2)^{\frac{4-p}{2}}}
\\
&\hspace{1cm}+\int\limits_{\O}\frac{\a +(p-1)u^2}{(\a +u^2)^{\frac{4-p}{2}}}z_1\,z_2\\
&\hspace{0.5cm}-\int\limits_{\O}f'\left(\left(\a +(u^+)^s\right)^{1/s}\right) 
\frac{(u^+)^{2s-2}}{\ \left(\a +(u^+)^s\right)^{\frac{2s-2}{s}}}z_1\,z_2\\
&-\int\limits_{\O}\frac{f\bigl(\left(\a +(u^+)^s\right)^{1/s}\bigr)\a (s-1)(u^+)^{s-2}}{\ \left(\a +(u^+)^s\right)^{\frac{2s-1}{s}}}z_1\,z_2.
\end{split}
\]

In addition, we introduce $Q^{\a}_u:W^{1,2}_0(\O)\to \R$ defined by
\[Q^{\a}_u(z)=B_\a(u)(z,z).\]
The definition of $B_\a(u)$ is inspired by the formal second derivative of $J_{\varepsilon,\a,h}$ in $u$.
Let us point out that, as $p < 2$, for any $u \in W^{1,\infty}(\O)$,  $B_\a(u)$ and $Q^{\a}_u$ are well defined on $W^{1,2}_0(\O)$, but not on $\w$. 
\newline In particular, $Q^{\a}_{u_0}$ is a smooth quadratic form on $W^{1,2}_0(\O)$ and we
define the Morse index of $J_{\varepsilon,\a ,h}$ at $u_0$ (denoted by $m(J_{\varepsilon,\a ,h} , u_0))$ as the supremum of the dimensions
of the linear subspaces of $W^{1,2}_0(\O)$ where $Q^{\a}_{u_0}$ is negative definite. Analogously, the large Morse index of $J_{\varepsilon,\a ,h}$ at $u_0$ (denoted by $m^{*}(J_{\varepsilon,\a ,h} , u_0))$ is the supremum of the dimensions of the linear subspaces
of $W^{1,2}_0(\O)$ where $Q^{\a}_{u_0}$ is negative semidefinite. We clearly have $m(J_{\varepsilon,\a ,h} , u_0) \leq m^{*}(J_{\varepsilon,\a ,h} , u_0) <
+\infty$.
This notion of Morse index is crucial in order to get estimates of the critical groups.

\bigskip

Indeed, the following result gives a description of the critical groups of the functional $J_{\varepsilon,\a ,h}$ at $u_0$ in terms of the Morse index. The proof derives directly from
\cite[Theorem~2.3]{cdv} (see also \cite[Theorem 1.3]{cdvlin}).
\bigskip

\begin{Theorem}\label{cornerstone}
	Let $h\in C^1(\bar{\Omega})$ and $\varepsilon,\, \alpha > 0$.
	If $u_0$ is a critical point of $J_{\varepsilon,\a ,h}$ and \[m(J_{\varepsilon,\a,h},u_0)=m^*(J_{\varepsilon,\a ,h},u_0),\]
	then $u_0$ is an isolated critical point of $J_{\varepsilon,\a ,h}$ and 
	
	\[
	\left\{
	\begin{array}{ll}
	C_m(J_{\varepsilon,\a ,h},u_0) \approx \K
	&\text{if $\ m = m(J_{\varepsilon,\a,h},u_0)$}\,,\\
	\noalign{\medskip}
	C_m(J_{\varepsilon,\a ,h}, u_0) = \{0\}
	&\text{if $\ m \neq m(J_{\varepsilon,\a ,h},u_0)$}\,.
	\end{array}
	\right.
	\]
	
\end{Theorem}

\bigskip

\begin{Rmk}\label{mu0}
	If the assumptions of the previous theorem are satisfied, the multiplicity of $u_0$ is one, namely, according to Definition~\ref{poincare}, $\ i(J_{\varepsilon,\a,h},u_0)(1)=1$.
\end{Rmk}

In order to prove Theorem~\ref{finale}, we recall an abstract theorem, proved in \cite{CLV} (see also \cite{benci} and \cite{chang}).

\begin{Theorem}\label{inter}
	Let $A$ be an open subset of a Banach space $X$. Let $f$  be a $C^1$
	functional on $A$ and $u \in A$ be an isolated critical point of
	$f$. Assume that there exists an open neighborhood $U$ of $u$ such
	that $\overline U \subset A$, $u$ is the only critical point of $f$
	in $\overline U$ and $f$ satisfies the Palais--Smale condition in
	$\overline U$.\par \noindent Then there exists $\bar \mu >0$ such
	that, for every $g \in C^1(A, \R)$ satisfying
	\begin{itemize}
		\item{} $\|f-g\|_{C^1(A)} < \bar \mu$,
		\item{}
		$g$ satisfies the Palais--Smale condition in $\overline U$,
		\item{}
		$g$ has a finite number $\{u_1,u_2,\dots,u_m \}$ of critical points
		in $U$,
	\end{itemize}
	\noindent we have
	\[
	\sum_{j=1}^{m} i(g,u_j)(t) = i(f,u)(t) + (1+t)Q(t),
	\]
	where $Q(t)$ is a formal series with coefficients in $\N \cup \{+
	\infty\}$.
\end{Theorem}

\section{Interpretation of multiplicity: number of distinct solutions of approximating problems}

Let $\e^*$ be
defined by Theorem \ref{aim} and $\e \in (0,\e^*)$. If $(P_\e)$ has
at least $\1$ distinct solutions, then the assert is proved, otherwise $\ie$ has a finite number of isolated critical points
$\overline u_1, \dots \overline u_k$ having multiplicities
$\overline m_1, \dots \overline m_k$ where
\[2\leq k < \1 \qquad \hbox{and} \qquad \sum_{j=1}^k \overline m_j \geq \1.\]
Let $\a_n \to 0^+$ and $R>0$ be such that
$\overline{B_R(\overline{u}_i)} \cap \overline{B_R(\overline{u}_j)}=\emptyset$, when $i\neq j$. We set
\begin{equation}\label{a}
	A=\bigcup^{k}_{j=1}B_{R}(\overline{u}_j)\,.
\end{equation}

If $\jn$, defined by (\ref{ja}), has
at least $\1$ critical points, then we
choose $h_n=0$, otherwise $\jn$ has $k_n<\1$ isolated critical
points $u_1,\dots u_{k_n}$, having multiplicities $m_1,\dots m_{k_n}$.
For simplicity, we will often omit the dependence on~$n$ of $u_i$ and their related objects. If $n$ is sufficiently large, by (\ref{vicinalfa}) and Theorem~\ref{inter}, $k_n\geq k$ and

\begin{equation}\label{summ}
\sum_{i=1}^{k_n} m_i\geq \sum_{j=1}^k \overline m_j \geq \1.
\end{equation}

 Reasoning as in \cite{cvbn} (p.11-12),
 we obtain that, for any $i=1,\dots k_n$, there are $V_i$ and $W_i$ subspaces of $\w$ 
such that

\begin{enumerate}
  \item  $\w=V_i\oplus W_i$;
  \item $V_i\subset C^1(\overline \O)$ 
  and $\, \dim V_i=m^*(\jn,u_i)<+\infty$;
  \item $V_i$ and $W_i$ are orthogonal in $L^2(\O)$;
  \item 
  $Q^{\a_n}_{u_i}(w) > 0 \quad$
  for any $w\in W_i\setminus\{0\}$.
  \end{enumerate}

\bigskip

Setting
\begin{equation}\label{vw}
V^n=V_1+V_2+\dots +V_{k_n} \quad \text{and}
\quad W^n=\bigcap_{i=1}^{k_n}W_i,
\end{equation}
we still have:

\begin{itemize}
	\item $\w=V^n\oplus W^n\,$;
	\item $V^n\subset C^1(\overline{\O})\ $ is finite dimensional \ and \ $W^n\subset W_i$ for any $i=1,\dots k_n$;
	\item $\int_{\O}vw =0$  for any $v\in V^n$, $w\in W^n$.
\end{itemize}

 \bigskip
 
\begin{Rmk}\label{nbar}
	We see that $\dim V^n\geq 1$, otherwise, for each $i=1,\dots k_n$, it should be $0=m^*(\jn,u_i)=m(\jn,u_i)\ \Rightarrow \ m_i=1$, so by (\ref{summ}) $k_n\geq \1$, while we are supposing $k_n<\1$.
 \end{Rmk}
 
 \bigskip
 
 In this setting the following two results hold (see Theorem 3.8 and Lemma 3.9 in \cite{cvbn}).
 
 \begin{Theorem}
 	There exist $r,\delta, M>0$, $\beta \in (0,1]$ and $\rho \in (0,r]$ such that for any $i\in \{1,\dots k_n\}$ and
 	$v\in V^n \cap \overline{B_\rho(0)}$
 	there exists one and 
 	only one $\psi_i(v)
 	\in W^n\cap B_r(0)$ such that
 	\begin{equation}\label{44}
 		\langle \jn'(u_i+v+\psi_i(v)),w\rangle=0 \qquad \forall w\in W^n.
 	\end{equation}
 	moreover $v+\psi _i(v)\in C^{1,\beta}(\overline{\Omega})$, $\|v+\psi_i(v)\|_{C^{1,\beta}(\overline{\Omega})}\leq M$ and, denoting by
 	\[
 	U_i=u_i+(V^n\cap B_\rho(0))+(W^n\cap B_r(0)),
 	\] 
 	we have $\overline{U_i}\cap \overline{U_j}=\emptyset$ if $i\neq j$
 	and $\bigcup_{i=1}^{k_n}\overline{U_i}\subset A$,  where $A$ is the open bounded set defined by (\ref{a}).
 	\newline Finally 
 	\begin{equation}\label{tubu}
 	\bn(u_i+v+\psi_i(v))(w,w)\geq \delta \int\limits_{\Omega}|\nabla w|^2dx
 	\end{equation}
 	for every $i\in \{1,\dots k_n\}$, $v\in V^n \cap \overline{B_\rho(0)}$ and  $w \in W^n\cap \h$.
 \end{Theorem}

\begin{Lem}
	\label{thm:phiC2}
	For any $i=1,\cdots k_n$, $\psi_i$ is continuous from
	$V^n\cap \overline{B_\rho(0)}$ in $W^n \cap C^{1}(\overline{\Omega})$ and of class $C^1$
	into $\h$. In addition,
	\begin{equation}\label{psip}
	\bn(u_i+z+\psi_i(z))(h+\langle\psi_i'(z),h\rangle,w)=0
	\end{equation}
	for any $z \in V^n\cap \overline{B_\rho(0)}$, $h\in V^n$ and
	$w \in \h$.
	
	Moreover 
	the function  $\varphi_i:V^n\cap \overline{B_\rho(0)} \to \R$  defined by
	\[\varphi_i(v)=\jn(u_i+v+\psi_i(v))\]
	is of class $C^2$	and
	\begin{equation}\label{fip}
	\langle \varphi'_i(z),h\rangle=\langle \jn'(u_i+z+\psi_i(z)),h\rangle
	\end{equation}
	\begin{equation}\label{fis}
	\langle\varphi''_i(z) h,v\rangle= B_{\alpha_n}
	\bigl(u_i+z+\psi_i(z)\bigr)(h+\psi'_i(z)h,v)\end{equation}
	for any $z \in V^n\cap \overline{B_\rho(0)}$
	and $h,v\in V^n$.
\end{Lem}

\vspace{1cm}

Our aim is to build a suitable perturbation of $\jn$,  such that all its critical points have multiplicity one.
\newline Let us denote by $V=V^n$ and $W=W^n$ the spaces introduced in (\ref{vw}) and let $\{e_1, \dots e_{\overline{n}}\}$ be an
$L^2$-orthonormal basis of
 $V$, where $\overline{n}=\dim V\geq 1$, as seen in Remark \ref{nbar}. Denoting by $V'$ the dual space of $(V, \|\ \|_{L^2})$, for any $v' \in V'$ we introduce $l_{v'}=\sum_{k=1}^{\nb}\langle v',e_k \rangle e_k \in V$ and
 $L_{v'}:\w \to \mathbb{R}$
 the functional defined by
 \[L_{v'}(u)=\int_\O l_{v'}u \, dx\, .\]

 By construction, for any $i\in \{1, \dots, k_n\}$ $u_i$ is the only critical point of $\jn$ in $U_i$ and $\overline{U_i}\subset A$. So let $\bar\m_i$ be defined by Theorem \ref{inter}
 and put $\m=\min\{\bar\m_1, \dots \bar\m_{k_n}\}$. 
 
 We prove that there is $\gamma >0$ such that
 
 \begin{equation}\label{normav}
 	v'\in V'\  \hbox{and} \ \|v'\|_{V'}<\gamma \quad \Rightarrow \quad  
 	\|l_{v'}\|_{C^1(\overline{\O})}<1/n \quad \hbox{and} \quad 
 	\|L_{v'}\|_{C^1(A)}< \mu.
 \end{equation}

 In fact, as  $V$ is finite dimensional, there is $c_n>0$ such that
 \[ 
  \|v\|_{C^1(\overline{\O})} \leq c_n \|v\|_{L^2} \qquad \forall v\in V.	
 	\]
 As $v'\in V'\!$, we have   \[\|l_{v'}\|_{C^1(\overline{\O})}=\|\sum_{k=1}^{\nb}\langle v',e_k \rangle e_k \|_{C^1(\overline{\O})}
 \leq \sum_{k=1}^{\nb}\| v'\|_{V'} \|e_k\|_{L^2}\, \|e_k \|_{C^1(\overline{\O})}
 \leq \nb c_n\| v'\|_{V'}\,.\]
 
 Moreover, $A$ being bounded,  there is $c_A>0$ such that
 \[\|L_{v'}\|_{C^1(A)}\leq c_A \|l_{v'}\|_{C^1(\overline{\O})}\,.
 \]

Therefore we get (\ref{normav}) by choosing $ \gamma= 
{\frac{1}{\nb c_n}\min \left\{\frac{1}{n},\, \frac{\mu}{c_A}\right\}}$.

\smallskip

 Set $\gamma_1= \gamma/k_n
 $. Applying Sard's Lemma to $\varphi_i':V\to V'$, there exists
 $v'_1\in V'$ such that  $\|v'_1 \nv<\g_1$ and if $\varphi'_1(v)=v'_1$,
 then $\varphi''_1(v)$ is an isomorphism.
 Moreover there is $\b_1>0$ such that if $v'\in V'$, $\|v'\nv \leq
 \b_1$ and $\varphi'_1(v)=v'_1 +v'$, then $\varphi''_1(v)$ is an isomorphism.

 \noindent Analogously, for $i=2, \dots k_n$, there exist $\b_i>0$,
 $\g_i=\min\{\g_{i-1},  \b_{i-1}/(k_n-i+1)\}$ and $v'_i\in V'$
 such that  $\|v'_i \nv<\g_i$ and if $v'\in V'$, $\|v'\nv \leq \b_i$ and
 $\varphi'_i(v)=v'_1+ \dots v'_i +v'$, then $\varphi''_i(v)$ is an isomorphism.
 
 So, denoting by $\overline{v}'_n=v'_1+ \dots v'_{k_n}$, $\ h_n=l_{\overline{v}'_n}\ $
 and $\ J_n=J_{\varepsilon,\a _n,\,h_n}$,  (\ref{normav}) shows that
 \[\|\overline{v}'_n\|< 
 \g \ \Rightarrow \ \|h_n\|_{C^1(\overline{\O})} 
 <1/n \quad \text{and} \quad \|J_n-\jn\|_{C^1(A)} 
 <\mu.
 \]
 
Solutions to 

\[ \ (P_n)\ \left\{
\begin{array}{ll}
- \ge^p div \bigl( (|\nabla u|^2 + \alpha_n)^{(p-2)/2} \nabla u \bigr)
+
u\,(\a _n+u^{2})^{(p-2)/2
}&  \\

=\fn (u) +h_n
& \hbox{in} \ \Omega \\

u>0 & \hbox{in} \ \Omega \\
u=0 & \hbox{on} \ \partial \Omega
\end{array}
\right.
\]

are critical points of the functional $J_n$.

\bigskip

We will prove that, when $\tilde{u}\in \w$ and $\bar{z}\in \h$, we have

\begin{equation} \label{+clm}
	 J'_n(\tilde{u})=0 \quad  \text{and} \quad B_{\a _n}(\tilde{u})(\bar{z}, \cdot)=0
	\qquad \Rightarrow \qquad \bar{z}=0.
\end{equation}

Indeed, since $\tilde{u}$ is a critical point of $J_n$, (\ref{vicinh}) and Corollary~\ref{lps} assure that $\tilde{u} \in U_i$, for a suitable $i\in \{1, \dots k_n\}$  and for $n$ large enough. In particular, there are $\tilde v \in V\cap B_\rho(0)$ and $\tilde w \in W\cap B_r(0)$ such that $\tilde u=u_i+\tilde v+\tilde w$.
\newline Considering that $V$ and $W$ are orthogonal in $L^2(\O)$,  
\[\langle\jn'(u_i+\tilde v+\tilde w),w\rangle=\langle J'_n(\tilde u),w\rangle +\int h_nw=0 \qquad \forall w \in W.\]
Therefore, by (\ref{44}), $\tilde{w}=\psi_i(\tilde{v})$.

\medskip

For each arbitrary $v\in V$
\[\int_{\O}h_nv =\int_{\O}\sum_{k=1}^{\nb}\langle \overline{v}'_n,e_k \rangle e_k v
=\langle \overline{v}'_n,\sum_{k=1}^{\nb}(\int_{\O}e_k v)e_k\rangle
=\langle \overline{v}'_n,v\rangle
\]
thus, by (\ref{fip}),
$
\langle \varphi'_i(\tilde{v}),v\rangle=\langle \jn'(\tilde{u}),v\rangle
=\langle J'_n(\tilde{u}),v\rangle + \int_{\O}h_nv
=\langle \overline{v}'_n,v\rangle
$.
Hence
$\ \varphi'_i(\tilde{v})=\overline{v}'_n=v'_1+ \dots v'_{k_n}$ and, by construction,
\begin{equation}\label{is}
\varphi''_i(\tilde v) \hbox{ is an isomorphism}.
\end{equation}

Let us write $\bar z=\bar v+\bar w$, where $\bar v \in V$ and $\bar w \in W$.
If $h\in V$, combining (\ref{fis}) and (\ref{psip}),  we get
\[\langle \varphi''_i(\tilde{v})h,\bar v\rangle
=\bn(\tilde u)\bigl(h+\langle\psi'_i(\tilde v),h\rangle,\bar v)
=\bn(\tilde u)\bigl(h+\langle\psi'_i(\tilde v),h\rangle,\bar z).
\]
As we are assuming that $B_{\a _n}(\tilde{u})(\bar{z}, \cdot)=0$,
\[\langle \varphi''_i(\tilde{v})h,\bar v\rangle=0 \qquad \quad \forall h\in V\] 
then, by (\ref{is}),  $\bar v$ must be $0$, so that $\bar z=\bar w \in W$.

Recalling (\ref{tubu}), we see that
\[0=B_{\a _n}(\tilde{u})(\bar{w}, \bar{w})\geq \delta \int\limits_{\Omega}|\nabla \bar{w}|^2,\]
hence $\bar w=0$ and this proves (\ref{+clm}).

\bigskip

In other words, if $\tilde{u}$ is a critical point of $J_n$, then $m^*(J_n,\tilde{u})=m(J_n,\tilde{u})$.
Consequently we can apply Theorem \ref{cornerstone},  getting that the multiplicity of every critical point of $J_n$ is one. Hence, by Theorem~\ref{inter} and (\ref{summ}), $J_n$ has at least $\tilde u_n^1,..,
\tilde u_n^{\1}$ distinct critical points.
 It remains to be proved that any
 $\tilde u_n^i$ is positive.
 From Theorem~\ref{c1reg} the critical points of $J_n$ are uniformely bounded in $C^{1, \eta}(\overline
 \Omega)$, thus, up to subsequences, $\tilde u_n^i$ converges in
 $C^1(\overline \Omega)$ to $\bar u_j$, for $n \to +\infty$. As $\bar u_j$ solves $(P_\e)$, by  Theorem 5 in \cite{vazquez} we infer
 that $\bar u_j>0$ and $\frac{\partial \bar u_j}{\partial \nu}(x_0)
 >0$, where $x_0 \in \partial \Omega$ and $\nu$ is the interior normal to $\partial \O $ at $x_0$.
 This implies   $\tilde u_n^i > 0$ on $\Omega$, for $n$ sufficiently large.\hfill$\Box$


\bigskip
 
\bigskip


\begin{thebibliography}{0}


\bibitem{ad}
S.~Almi and M.~Degiovanni,
On degree theory for quasilinear elliptic equations with natural
growth conditions, in \textit{Recent Trends in Nonlinear Partial
	Differential Equations II: Stationary Problems} (Perugia, 2012),
J.B.~Serrin, E.L.~Mitidieri and V.D.~R\u{a}dulescu eds.
\textit{Contemporary Mathematics} {\bf 595} (2013), 1--20
	

\bibitem{alves}
A. Alves, Existence and multiplicity of solution for a class of
quasilinear equations, {\it Advances Nonlinear Studies} {\bf 5}
(2005), 73--87.


\bibitem{AP1} J.G. Azorero and I. Peral, Existence and nonuniqueness for the p-Laplacian: Nonlinear eigenvalues, {\it Commun. Partial Differ. Equations} {\bf 12} (1987), 1389--1430.

\bibitem{AP2} J.G. Azorero and I. Peral, Multiplicity of solutions for elliptic problems with critical exponents or with a symmetric term, {\it Trans. Amer. Math. Soc.} {\bf 323} (1991), 877--895.

\bibitem{benci}
V.~Benci, A new approach to the Morse-Conley theory and some
applications, {\it Ann. Mat. Pura Appl.}  {\bf 158} (1991), 231--305.


\bibitem{bencicerami}
V.~Benci and G.~Cerami, Multiple positive solutions of some elliptic
problems via the Morse theory and the domain topology, {\it Calc.
Var. P.D.E.} {\bf 2} (1994), 29--48.


\bibitem{brezisnirenberg} H.~Brezis and L.~Nirenberg, Positive solutions of nonlinear elliptic  equations involving critical Sobolev exponents, {\it Comm. Pure Appl. Math.} {\bf 36} (1983), 437--477.

\bibitem{browder1983}
F.E.~Browder, Fixed point theory and nonlinear problems, {\it Bull. Amer. Math. Soc. (N.S.)} \textbf{9} (1983), no. 1, 1--39.

\bibitem{chang1} K. Chang, Morse Theory on Banach space and its
applications to partial differential equations,  {\it Chin. Ann. of
Math.} {\bf 4B} (1983), 381--399.


\bibitem{chang}
K.C.~Chang, Infinite dimensional Morse theory and multiple solution
problems, Birkh\"auser, Boston-Basel-Berlin, 1993.

\bibitem{chang3}
K.C.~Chang, Morse theory in nonlinear analysis, in Nonlinear
Functional Analysis and Applications to Differential Equations, A.
Ambrosetti, K.C. Chang, I. Ekeland Eds., World Scientific Singapore,
1998, 60--101.


\bibitem{CD}
S.~Cingolani and M.~Degiovanni, {Nontrivial solutions for $p$-Laplace
equations with right hand side having $p$-linear growth at infinity}, {\it Comm. Partial Differential Equations} {\bf 30} (2005), 1191--1203.

\bibitem{cdvlin}
S.~Cingolani, M.~Degiovanni and G.~Vannella, {On the critical polynomial of functionals related to $p$-area $(1 < p < + \infty)$ and $p$-Laplace $(1 < p \leq 2)$ type operators}, {\it Atti Accad. Naz. Lincei Rend. Lincei Mat. Appl.} {\bf 26}
(2015), 49--56.

\bibitem{cdv}
S.~Cingolani, M.~Degiovanni and G.~Vannella, {Amann-Zehnder type results for $p$-Laplace problems},
{\it Annali di Matematica Pura ed Applicata},  {\bf 197}
(2018), 605--640.

\bibitem{CLV}
S.~Cingolani, M.~Lazzo and G.~Vannella, Multiplicity results for a
quasilinear elliptic system via Morse theory, {\it Commun. Contemp.
Math.} {\bf 7} (2005), 227--249.


\bibitem{CV}
S.~Cingolani and G.~Vannella, Critical groups computations on a class
of Sobolev Banach spaces via Morse index, {\it Ann. Inst. H.
Poincar\'e Anal. Non Lin\'eaire} {\bf 20} (2003), 271--292.


\bibitem{cvmp}
S.~Cingolani and G.~Vannella, Marino-Prodi perturbation type results
and Morse indices of minimax critical points for a class of
functionals in Banach spaces, {\it Ann. Mat. Pura Appl.}, {\bf 186}
(2007), 157--185.



\bibitem{cvjde}
S.~Cingolani and G.~Vannella, On the multiplicity of positive solutions for p-Laplace equations via Morse theory, {\it J. Differential Equations} {\bf 247} (2009), 3011--3027.

\bibitem{cvbn}
S.~Cingolani and G.~Vannella, The Brezis-Nirenberg type problem for the $p$-laplacian $(1<p<2)$: Multiple positive solutions, {\it J. Differential Equations} {\bf 266} (2019), 4510--4532.

\bibitem{DL}
M.~Degiovanni and S.~Lancelotti, Linking over cones and nontrivial
solutions for p-Laplace equations with p-superlinear nonlinearity,
{\it Ann. Inst. H. Poincar\'e Anal. Non Lin\'eaire} {\bf 24} (2007),
907--919.

\bibitem{doomed}
J.M. Do \'{O} and E. S. Medeiros, {Remarks on least energy solutions for quasilinear elliptic problems in $\R^N$}, Electronic J. Diff. Eqs.  2003 (2003), pp. 1--14,


\bibitem{franchilanconelliserrin}
B.~Franchi, S.~Lanconelli and J. Serrin, Existence and uniqueness of nonegative solutions of quasilinear equations in $\R^N$, {\it Adv. Math.} {\bf 118} (1996), 177--243.

\bibitem{gueddaveron}
M.~Guedda and L.~Veron, Quasilinear elliptic equations involving critica Sobolev exponents, {\it Nonlinear Anal.} {\bf 13} (1989),
879--902.


\bibitem{lieberman}
G.M.~Lieberman, Boundary regularity for solutions of degenerate
elliptic equations, {\it Nonlinear Anal.} {\bf 12} (1988),
1203--1219.


\bibitem{MarPro}
A.~Marino and G.~Prodi, Metodi perturbativi nella teoria di Morse, {\it Boll. U.M.I.} (4) {\bf 11} Suppl. fasc. 3 (1975), 1--32.

\bibitem{skrypnik1994}
I.V.~Skrypnik, Methods for analysis of nonlinear elliptic boundary value problems,
{\it Translations of Mathematical Monographs}, \textbf{139},
American Mathematical Society, Providence, RI, 1994.


\bibitem{Uhlen}
K.~Uhlenbeck, Morse theory on Banach manifolds, {\it J. Funct. Anal.} {\bf 10} (1972), 430--445.

\bibitem{vazquez}
J.L~Vazquez, A strong maximum principle for some quasilinear elliptic equations, {\it Appl. Math. Optim.} {\bf 12} (1984), 191--202.

\end{thebibliography}
\end{document}